\documentclass[a4paper,10pt]{amsart}

\usepackage{color, graphics}

\usepackage{graphicx}

\usepackage[T1]{fontenc}
\usepackage[latin1]{inputenc}
\usepackage[english]{babel}
\usepackage{geometry}
\usepackage{indentfirst} 

\usepackage{amsmath,amsfonts,amssymb,amsthm}
\usepackage{cases}
\usepackage{latexsym}

\makeatletter
\def\@currentlabel{2.1}\label{e:dispaa}
\def\@currentlabel{2.21}\label{e:dispau}
\def\@currentlabel{2.22}\label{e:dispav}
\def\@currentlabel{2.23}\label{e:dispaw}
\def\@currentlabel{2.24}\label{e:dispax}
\def\theequation{\thesection.\@arabic\c@equation}
\makeatother

\newcommand{\R}{\mathbb{R}}

\newcommand{\eps}{\varepsilon}

\renewcommand{\theequation}{\thesection.\arabic{equation}}
\newtheorem{theorem}{Theorem}
\newtheorem{proposition}[theorem]{Proposition}
\newtheorem{lemma}[theorem]{Lemma}
\newtheorem{corollary}[theorem]{Corollary}
\newtheorem{remark}[theorem]{Remark}

\newtheorem{question}[theorem]{Question}

\begin{document}

\title{Radial and cylindrical symmetry of solutions to the Cahn-Hilliard equation}

\author{Matteo Rizzi}
\address{Centro
de Modelamiento Matem\'atico (UMI 2807 CNRS), Universidad de Chile, Casilla
170 Correo 3, Santiago, Chile.}
\email{mrizzi1988@gmail.com}
\thanks{The author was partially supported by Fondecyt postdoctoral research grant 3170111 and Fondo Basal AFB170001 CMM-Chile. Moreover, the author is particularly grateful to his PhD advisor Alberto Farina for his crucial comments and remarks about this paper}

\begin{abstract}
The paper is devoted to the classification of entire solutions to the Cahn-Hilliard equation $-\Delta u=u-u^3-\delta$ in $\R^N$, with particular interest in those solutions whose nodal set is either bounded or contained in a cylinder. The aim is to prove either radial or cylindrical symmetry, under suitable hypothesis.
\end{abstract}

\maketitle

\textbf{AMS classification: 35B10, 35B06}

\section{Introduction}

\setcounter{equation}{0}

We consider the entire equation
\begin{eqnarray}
-\Delta u=f(u)-\delta &\text{in $\R^N$,}\label{cahn-hilliard}
\end{eqnarray}
with $f(u):=u-u^3$ and $\delta\in\R$. This equation has a variational characterisation, indeed, if we consider it on a domain $\Omega\subset\R^N$, it arises as the Euler equation of the Ginzburg-Landau functional
\begin{equation}
E(u,\Omega):=\int_\Omega \bigg(\frac{1}{2}|\nabla u|^2+W(u)\bigg) dx, \qquad W(t):=\frac{(1-t^2)^2}{4},\label{def_energy}
\end{equation} 
under the mass constraint
\begin{eqnarray}
\frac{1}{|\Omega|}\int_\Omega u dx=m, \qquad m\in(-1,1),\label{mass_constraint}
\end{eqnarray}
which gives rise to the Lagrange multiplier $\delta$. The interest in the minimisers $u$ of $E(\cdotp,\Omega)$ arises from the phase transitions theory. In other words, if two different fluids are mixed in a container $\Omega$, the number $u(x)$ represents the density of one of the two at $x$, in an equilibrium configuration. Here we take $\delta\in(-\frac{2}{3\sqrt{3}},\frac{2}{3\sqrt{3}})$, so that the polynomial $f_\delta(t):=t-t^3-\delta$ admits exactly $3$ real roots $$z_1(\delta)<-1/\sqrt{3}<z_2(\delta)<1/\sqrt{3}<z_3(\delta),$$ 
with $z_2(\delta)$ satisfying $\delta z_2(\delta)\ge 0$. The main results of the paper deal with symmetry properties of entire solutions to the Cahn-Hilliard equation (\ref{cahn-hilliard}). 
\begin{theorem}
Let $N\ge 2$, $\delta\in(-\frac{2}{3\sqrt{3}},\frac{2}{3\sqrt{3}})$ and let $u_\delta$ be a solution to (\ref{cahn-hilliard}) such that
\begin{equation}
u_\delta>z_2(\delta) \qquad \text{outside a ball $B_R\subset\R^N$.}\label{u_pos_ball} 
\end{equation}
\begin{enumerate}
\item If $\delta\in(-\frac{2}{3\sqrt{3}},0]$, then $u\equiv z_3(\delta)$. \label{const_case}
\item If $\delta\in(0,\frac{2}{3\sqrt{3}})$, then $u_\delta$ is radially symmetric (not necessarily constant).\label{rad_case}
\end{enumerate}
\label{th_global_radial}
\end{theorem}
We note that, for $\delta>0$, nontrivial bubble solutions are known to exist. This is an important difference with the case $\delta\le 0$. Moreover, we will see that the zero level set of radial solutions is non empty. In particular, we have the following Corollary. 
\begin{corollary}
Let $\delta\in (0,\frac{2}{3\sqrt{3}})$ and let $u_\delta$ be a non constant solution to (\ref{cahn-hilliard}) such that $u_\delta>z_2(\delta)$ outside a ball $B_R$. Then the nodal set of $u_\delta$ is a sphere.
\label{cor_level_set}
\end{corollary}
This result agrees with the variational theory, which studies the asymptotic behaviour of the scaled functionals
\begin{equation}
E_\eps(u,\Omega)=\int_\Omega \bigg(\frac{\eps}{2}|\nabla u|^2+\frac{W(u)}{\eps}\bigg) dx
\end{equation}
as $\eps\to 0$. For instance, Modica proved that, if $\eps_k$ is a sequence of positive numbers tending to $0$ and $u_{\eps_k}$ is a sequence of minimisers of $E_{\eps_k}(\cdotp,\Omega)$ under the constraint (\ref{mass_constraint}) such that $u_{\eps_k}\to u_0$ in $L^1(\Omega)$, then $u_0(x)\in\{\pm 1\}$ for almost every $x\in\Omega$, and the boundary in $\Omega$ of the set $E:=\{x\in\Omega:\, u_0(x)=1\}$ has minimal perimeter among all subsets $F\subset \Omega$ such that $|F|=|E|$, where $|\cdotp|$ denotes the volume (see \cite{M}, Theorem $1$). Further $\Gamma$-convergence results relating $E_\eps(\cdotp,\Omega)$ to the perimeter can be found in \cite{MM}. Therefore, given a family $\{u_\eps\}_{\eps\in(0,\eps_0)}$ of minimisers under the constraint (\ref{mass_constraint}), their nodal set is expected to be close to a compact Alexandrov-embedded constant mean curvature surface, at least for $\eps$ small. Corollary \ref{cor_level_set}, together with a scaling argument, shows that, for $\eps$ small enough, the nodal set of \textit{any} entire solution to
\begin{equation}
-\eps\Delta u=\eps^{-1}(u-u^3)-\ell,\qquad\ell>0,\label{cahn-hilliard-scaled}
\end{equation}
in $\R^N$ such that $u>z_2(\eps \ell)$ outside a ball is actually a sphere, which is known to be the unique compact Alexandrov-embedded constant mean surface in $\R^N$ (see \cite{A}).\\

After that, we set
$$C_R:=\{(x',x_N)\in\R^N:\, |x'|<R\}$$
and we consider solutions satisfying
\begin{equation}
u_\delta>z_2(\delta) \qquad \text{outside a cylinder $C_R\subset\R^N$.}\label{u_pos_cyl}
\end{equation}
The aim is to study their symmetry properties and their asymptotic behaviour as $\delta\to 0$, with particular interest in solutions which have one periodicity direction.
\begin{theorem}
Let $\{u_\delta\}_{\delta\in(0,\frac{2}{3\sqrt{3}})}$ be a family of non constant solutions to (\ref{cahn-hilliard}) in $\R^N$, with $N\ge 2$. Assume furthermore that $u_\delta$ is periodic in $x_N$ and, for any $\delta\in(0,\frac{2}{3\sqrt{3}})$, there exists $R(\delta)>0$ such that (\ref{u_pos_cyl}) is true. Then
\begin{enumerate}
\item $z_1(\delta)<u_\delta(x)<z_3(\delta)$, for any $x\in\R^N$.\label{u_bounded}
\item $u_\delta$ is radially symmetric in $x'$.\label{u_symm}
\item $u_\delta\to -1$ as $\delta\to 0$ uniformly on compact subsets of $\R^N$.\label{u_to-1}
\end{enumerate}
\label{main_th}
\end{theorem}
In view of the aforementioned $\Gamma$-convergence results, given a solution $u$ to (\ref{cahn-hilliard-scaled}) satisfying (\ref{u_pos_cyl}), with $\delta=\eps\ell$, we expect its nodal set to be close to an Alexandrov-embedded constant mean curvature surface which is contained in a cylinder. This kind of surfaces are fully classified, at least the ones which are embedded in $\R^3$, in fact it is known that the unique examples are the sphere and Delaunay \textit{unduloids}, that is a family of non compact revolution surfaces obtained by rotating a periodic curve around a fixed axis in $\R^3$, which can be taken to be the $x_3$-axis, parametrised by a real number $\tau\in(0,1)$. We will denote the period of $D_\tau$ by $T_\tau$. For a detailed introduction of Delaunay surfaces, we refer to \cite{JP,MP}. For any $\tau\in(0,1)$, Kowalczyk and Hernandez \cite{HK} constructed a family $\{u_{\tau,\eps}\}_{\eps\in(0,\eps_0)}$ of solutions to (\ref{cahn-hilliard-scaled}) in $\R^3$, with $\ell=\ell_\eps$ depending on $\eps$, such that 
\begin{enumerate}
\item $\ell_\eps$ is positive and bounded uniformly in $\eps$. \label{molt_lagr>0}
\item $u_{\tau,\eps}$ is radially symmetric in $x'$.\label{rad_x'}
\item $u_{\tau,\eps}(x)\to \pm 1$ as $\eps\to 0$, uniformly on compact subsets of $\Omega_\tau^\pm$, where $\Omega_\tau^\pm$ denote the exterior and the interior of the Delaunay surface $D_\tau$, respectively.\label{lim_eps}
\item $u_{\tau,\eps}(x',x_3)\to z_3(\eps\ell_\eps)$ as $|x'|\to\infty$, uniformly in $x_3$.\label{lim_x'}
\item $u_{\tau,\eps}$ is periodic in $x_3$ of period $T_\tau$. \label{u_per_xN}
\end{enumerate}
We observe that the solutions $u_{\eps,\tau}$ constructed in \cite{HK} are actually negative outside a cylinder, however, in order to obtain the aforementioned family, thanks to the oddness of $f$, it is enough to replace them with $-u_{\eps,\tau}$. An interesting question is uniqueness. 
In other words, we are interested in the following question.\\

\begin{question}[Uniqueness]
Let $\eps_0>0$, $\tau\in(0,1)$ and let $v$ be a non constant solution to (\ref{cahn-hilliard-scaled}) in $\R^3$ with $\ell=\ell_\eps$, for $\eps\in(0,\eps_0)$. Assume in addition that
\begin{itemize}
\item $\ell_\eps$ is bounded uniformly in $\eps$.
\item $v$ is periodic in $x_3$, with period $T_\tau$.
\item $v>z_2(\eps\ell_\eps)$ outside a ball $B_R$.
\end{itemize}
Is it true that $v=u_{\eps,\tau}$, at least if $\eps_0$ is small enough?
\end{question}

This would be the counterpart of Corollary \ref{cor_level_set} for periodic solutions. For now we are not able to give a full answer to this question. However Theorem \ref{main_th} is a first step in this direction, since it proves that any family $\{v_\eps\}_{\eps\in(0,\eps_0)}$ of such solutions has to share many properties with the family $\{u_{\tau,\eps}\}_{\eps\in(0,\eps_0)}$ constructed by Hernandez and Kowalczyk. For instance, for $\eps$ small, $v_\eps$ has to satisfy (\ref{molt_lagr>0}), (\ref{rad_x'}), (\ref{lim_x'}) and the scaled functions $v_\eps(\eps x)$ tend to $-1$ uniformly on compact subsets of $\R^N$ as $\eps\to 0$.\\

The plan of the paper is the following. In Section \ref{section_results} we will state some quite general results, of which the Theorems stated in the introduction are consequences. Section \ref{section_proofs} is devoted to the proofs. It is divided into three subsections, dedicated to prove global boundedness, radial symmetry and the asymptotic behaviour for $\delta$ small respectively. 

\section{Some relevant results}\label{section_results}

In this section we state some results that are proved in section \ref{section_proofs}. First we prove boundedness of solutions, which holds irrespectively of the sign of $\delta$.
\begin{proposition}
Let $\delta\in(-\frac{2}{3\sqrt{3}},\frac{2}{3\sqrt{3}})$ and let $u_\delta\in L^3_{loc}(\R^N)$ be a distributional solution to the Cahn-Hilliard equation (\ref{cahn-hilliard}). Then 
$$z_1(\delta)\le u_\delta(x)\le z_3(\delta)$$
a. e. in $\R^N$.
\label{prop_bounded}
\end{proposition}
\begin{remark}
\begin{itemize}
\item Using Proposition \ref{prop_bounded}, standard elliptic estimates (see \cite{GT}, Theorem $8.8$ and Corollary $6.3$) and a bootstrap argument, it is possible to show that any distributional solution $u\in L^3_{loc}(\R^3)$ is actually in $C^\infty(\R^N)$. This parallels the regularity result proved in \cite{F1} for the Allen-Cahn equation.
\item It follows from the strong maximum principle that either $u_\delta$ is constant, and in this case it has to be either $z_1(\delta)$, or $z_2(\delta)$ or $z_3(\delta)$, or it satisfies $z_1(\delta)<u_\delta<z_3(\delta)$ in $\R^N$.
\end{itemize}
\label{rem_regularity}
\end{remark}
We observe that Proposition \ref{prop_bounded} and Remark \ref{rem_regularity} prove point (\ref{u_bounded}) of Theorem \ref{main_th}, which is actually true for any non constant entire solution. After that, we rule out the case $\delta\le 0$, in which only constant solutions are allowed. 
\begin{proposition}
Let $u_\delta$ be a solution to (\ref{cahn-hilliard}) in $\R^N$, with $-\frac{2}{3\sqrt{3}}<\delta\le 0$ such that $u_\delta>z_2(\delta)$ outside a stripe $\{x\in\R^N:\,|x_N|<L\}$. Then $u_\delta\equiv z_3(\delta)$.
\label{prop_delta>0}
\end{proposition}
We stress that the latter result proves point (\ref{const_case}) of Theorem \ref{th_global_radial} and agrees with the sign of $\delta$ obtained by Hernández and Kowalczyk in \cite{HK}. Using boundedness and the famous result by Gidas, Ni and Nirenberg \cite{GNN}, or Theorem $2$ of \cite{FMR}, which relies on the moving planes method, we can prove this symmetry result. 
\begin{proposition}
Let $\delta\in(0,2/3\sqrt{3})$ and let $u_\delta$ be a non constant solution to (\ref{cahn-hilliard}) such that $u_\delta>z_2(\delta)$ outside a ball $B_R$, for some $R>0$. Then\\
\begin{itemize}
\item $u_\delta$ is radially symmetric, that is, up to a translation, $u_\delta(x)=w_\delta(|x|)$.
\item $u_\delta$ is radially increasing, in the sense that $(\nabla u_\delta(x),x)> 0$, for any $x\in\R^N\backslash\{0\}$.
\end{itemize}
\label{prop_rad_global}
\end{proposition}
Proposition \ref{prop_rad_global} proves point (\ref{rad_case}) of Theorem \ref{th_global_radial}. More precisely, it is known that, for $\delta\in(0,\frac{2}{3\sqrt{3}})$, the problem
\begin{equation}
\begin{cases}
-\Delta v_\delta=v_\delta-v_\delta^3-\delta &\text{in $\R^N$}\\
v_\delta(0)=\min_{\R^N} v_\delta, \, v_\delta<z_3(\delta), \, v_\delta(x)\to z_3(\delta) &\text{as $|x|\to\infty$}
\end{cases}
\label{prob_rad}
\end{equation}
admits a unique solution which is radially symmetric (see \cite{WW,D,PS}), that is $v_\delta(x)=w_\delta(|x|)$. In view of this fact, we can actually prove the following classification result.
\begin{proposition}
Let $\delta\in (0,\frac{2}{3\sqrt{3}})$ and let $u_\delta$ be a non constant solution to (\ref{cahn-hilliard}) such that $u_\delta>z_2(\delta)$ outside a ball $B_R$. Then, up to a translation, $u_\delta=v_\delta$.
\label{prop_uniqueness}
\end{proposition}
In the sequel, we will use the notation $W_\delta(t):=W(t)+\delta t$.
\begin{remark}
It is possible to see that, for any $\delta\in(0,\frac{2}{3\sqrt{3}})$, there exists $R(\delta)>0$ such that $w_\delta(R(\delta))=0$. In fact, the energy 
$$E_\delta(r):=\frac{1}{2}(w'_\delta)^2-W_\delta(w_\delta)$$
is strictly decreasing, since
$$\frac{d}{dr}E(r)=w'_\delta(w''_\delta-W'_\delta(w_\delta))=-\frac{N-1}{r}(w'_\delta)^2<0, \qquad\forall \,r>0.$$
Thus, using that, by Proposition \ref{prop_rad_global}, $v_\delta$ is decreasing,
\begin{equation}
-W_\delta(w_\delta(0))=E_\delta(0)>0,\label{w_delta_0<0}
\end{equation}
which yields that $w_\delta(0)<0$.
\label{remark_w_delta_0<0}
\end{remark} 
In particular, in view of Remark \ref{remark_w_delta_0<0}, which yields that the nodal set of $v_\delta$ is neither empty nor a singleton, Corollary \ref{cor_level_set} is true.\\

Considering solutions that are approaching a positive limit just with respect to $N-1$ variables, we can prove the following.
\begin{proposition}
Let $\delta\in(0,2/3\sqrt{3})$ and let $u_\delta$ be a non constant solution to (\ref{cahn-hilliard}) such that $u_\delta>z_2(\delta)$ outside a cylinder $C_R$, for some $R>0$. If $u_\delta$ is periodic in $x_N$, then\\
\begin{itemize}
\item $u_\delta$ is radially symmetric in $x'$, that is, up to a translation, $u_\delta(x)=w_\delta(|x'|,x_N)$.
\item $u_\delta$ is radially increasing, in the sense that $(\nabla u_\delta(x), (x',0))> 0$, for any $x=(x',x_N)\in\R^N\backslash\{0\}$.
\end{itemize}
\label{th_rad}
\end{proposition}
We note that this proves point (\ref{u_symm}) of Theorem \ref{main_th}. Even in this case, our result agrees with the construction of \cite{HK}, where the authors prove the existence of a family of solutions fulfilling the symmetries of the Delaunay surface $D_\tau$, hence, in particular they are periodic in $x_N$, radially symmetric and radially increasing in $x'$. Here we show that \textit{any} periodic solution has to be radially symmetric and radially increasing in $x'$. Finally, in order to prove point (\ref{u_to-1}) of Theorem \ref{main_th}, we need the following result, which shows that the phase transition has to be complete.
\begin{proposition}
For any $\epsilon>0$ there exists $\delta_0\in(0,\frac{2}{3\sqrt{3}})$ such that, for any $\delta\in(0,\delta_0)$ and for any non constant solution $u_\delta$ to (\ref{cahn-hilliard}) satisfying $\sup_{\R^N}u_\delta=z_3(\delta)$, we have 
\begin{equation}
\inf_{\R^N} u_\delta <-1+\epsilon.
\end{equation}
\label{prop_inf-sup}
\end{proposition}
This result somehow parallels Lemma $2.5$ of \cite{FV}. The proof relies on both the moving planes and the sliding method. For a detailed proof of point (\ref{u_to-1}) of Theorem \ref{main_th}, we refer to section \ref{section_proofs}.

\section{The proofs}\label{section_proofs}

\setcounter{equation}{0}

\subsection{Boundedness}

In order to prove boundedness for distributional solutions to (\ref{cahn-hilliard}), we will rely on a result proved by Brezis in \cite{B}.
\begin{lemma}[Brezis-Kato inequality]
Let $p>1$ and assume that $v\in L^p_{loc}(\R^k)$ satisfies
\begin{eqnarray}\notag
-\Delta v+|v|^p v\le 0 &\text{in $\mathcal{D}'(\R^N)$.}
\end{eqnarray}
Then $v\le 0$ a.e. in $\R^N$.
\label{lemma_BK}
\end{lemma}
Now we prove Proposition \ref{prop_bounded}.
\begin{proof}
Writing $-f_\delta(t)=(t-z_1(\delta))(t-z_2(\delta))(t-z_3(\delta))$ and setting 
\begin{eqnarray}\notag
\alpha:=z_1(\delta)-z_2(\delta)<0,\\\notag 
\beta:=z_3(\delta)-z_2(\delta)>0,\\\notag
w:=u_\delta-z_2(\delta),
\end{eqnarray}
we have
\begin{eqnarray}
\Delta w=\Delta u_\delta=(u_\delta-z_1(\delta))(u_\delta-z_2(\delta))(u_\delta-z_3(\delta))=w(w-\alpha)(w-\beta),\label{eq_w}
\end{eqnarray}
thus
\begin{eqnarray}\notag
\Delta (w-\beta)^+=\chi_{\{w>\beta\}}\Delta w=\chi_{\{w>\beta\}}w(w-\alpha)(w-\beta)\ge ((w-\beta)^+)^3,
\end{eqnarray}
where $\chi_{\{w>\beta\}}$ denotes the characteristic function of the set $\{\in\R^N:\, w(x)>\beta\}$. By the Kato-Brezis inequality (see Lemma \ref{lemma_BK}), we have $w\le \beta$. The same argument applied to $(\alpha-w)^+$ gives the lower bound $w\ge \alpha$. 
\end{proof}
\begin{remark}
A similar argument is used in \cite{F2} to prove boundedness for solutions to a class of vectorial equations of the form
$$\Delta u=u P'_n(|u|^2), \qquad P_n(t):=\frac{1}{2}\Pi_{j=1}^n (t-k_j)^2,$$
with $0<k_1<\dots<k_n$. The scalar Allen-Cahn equation is included in this class. Here we prove that a similar result is true for a slightly different non linearity, due to the presence of $\delta$.
\end{remark}
Now we can prove Proposition \ref{prop_delta>0}, using boundedness and a result of \cite{F1} where non-existence f ground states for some special non lineariries is proved.
\begin{proof}
By Lemma \ref{lemma_BK}, $z_1(\delta)\le u_\delta\le z_3(\delta)$, in particular, since $\delta\le 0$, $|z_1(\delta)|\le z_3(\delta)$, hence $|u_\delta|\le z_3(\delta)$. By Lemma \ref{lemma_dec_x'}, $u_\delta\to z_3(\delta)$ as $x_1\to\pm\infty$, the limit being uniform in $x'$. Moreover, setting $f_\delta(t):=f(t)-\delta$, we have 
\begin{itemize}
\item $f_\delta(t)\ge 0$, $\forall\,t\in(0,z_3(\delta))$,
\item $f_\delta(t)+f_{-\delta}(t)=-2\delta\ge 0$, $\forall\,t\in(0,z_3(\delta))$,
\item $f_\delta(t)$ is non increasing in a left neighbourhood of $z_3(\delta)$.
\end{itemize}
Therefore, by Theorem $4.2$ of \cite{F1}, $u_\delta\equiv z_3(\delta)$.
\end{proof}

\subsection{Radial symmetry}

The aim of this subsection is to prove Proposition \ref{th_rad}. In order to do so, we need some decay at infinity of the solution. From now on, we denote the variables by $x:=(x_1,x'')\in\R\times\R^{N-1}$. For $\lambda\in\R$, we set
\begin{eqnarray}
\Sigma_\lambda:=\{x\in\R^3:\, x_1<\lambda\}.\label{def_Sigma_lambda}
\end{eqnarray}
This changing of notation is justified by the fact that several times this section $x_N$ is the periodicity variable, hence we are not allowed to start the moving planes in that direction.
\begin{lemma}
Let $u_\delta$ be a solution to (\ref{cahn-hilliard}). Assume furthermore that $u_\delta>z_2(\delta)$ in the half-space $\R^N\backslash\Sigma_\lambda$, for some $\lambda\in\R$. Then 
\begin{eqnarray}
u(x_1,x'')\to z_3(\delta), &\text{as $x_1\to\infty$, uniformly in $x''$.}
\end{eqnarray}
\label{lemma_dec_x'}
\end{lemma}
\begin{proof}
The statement is trivial if $u_\delta$ is constant (see Remark \ref{rem_regularity}), hence we can assume that it is non constant. We apply Lemma $2.3$ of \cite{F1} to $w:=u_\delta-z_2(\delta)$ in the half space $\R^N\backslash \Sigma_\lambda$, where, by Lemma \ref{lemma_BK}, $0<w<\beta$. This is possible since the non linearity $g(t):=-t(t-\alpha)(t-\beta)$ is positive in $(0,\beta)$ and $g'(0)>0$. We recall that the constants $\alpha$ and $\beta$ are defined in the proof of Proposition \ref{prop_bounded}. The conclusion is that
\begin{eqnarray}\notag
w(x_1,x'')\to\beta &\text{as $x_1\to \infty$,}
\end{eqnarray}
and the limit is uniform in the other variables. 
\end{proof}
Using the fact that $f'(z_3(\delta))<0$, we can actually prove a better result about the decay rate of $z_3(\delta)-u_\delta$.
\begin{lemma}
Let $u_\delta$ be a solution to (\ref{cahn-hilliard}) such that $u_\delta>z_2(\delta)$ in the half space $\R^N\backslash \Sigma_\lambda$, for some $\lambda\in\R$. Then, for any $\gamma\in (0,\sqrt{-f'(z_3(\delta))})$, there exists a constant $C(\gamma)>0$, depending on $\gamma$, such that
\begin{eqnarray}
0<z_3(\delta)-u_\delta(x_1,x'')\le C(\gamma) e^{-\gamma x_1}, \qquad\forall\, x=(x_1,x'')\in \R^N\backslash\Sigma_\lambda.\label{exp_dec}
\end{eqnarray}
\label{lemma_exp_dec}
\end{lemma}
\begin{proof}
We compare the bounded function $v:=z_3(\delta)-u_\delta$ with the barrier $\mu e^{-\gamma x_1}$, for $\gamma\in (0,\sqrt{-f'(z_3(\delta))})$, in the half-space $\R^N\backslash\Sigma_M$, with $M>0$ large enough. In fact, on $\partial(\R^N\backslash\Sigma_M)$, we have
$$v(x)\le\|v\|_{L^\infty(\R^N)} \le\mu e^{-\gamma M},$$
provided $\mu\ge \|v\|_{L^\infty(\R^N)}e^{\gamma M}$. Note that here we use the fact that $v\in L^\infty$, which is true by Lemma \ref{lemma_BK}. Moreover, setting $h_\delta(v):=-f_\delta(z_3(\delta)-v)$, we have $h_\delta(0)=-f_\delta(z_3(\delta))=0$ and $h'_\delta(0)=f'(z_3(\delta))<0$, thus
$$(-\Delta+\gamma^2)(v-\mu e^{\gamma x_1})=h_\delta(v)+\gamma^2 v\le 0$$
in $\R^N\backslash\Sigma_M$ if $M$ is large enough, since, by Lemma \ref{lemma_dec_x'}, $z_3(\delta)-v$ is decaying as $x_1\to\infty$, uniformly with respect to $x''$. Thus, by the maximum principle for possibly unbounded domains (see Lemma $2.1$ of \cite{BCN}), we conclude that (\ref{exp_dec}) is true in $\R^N\backslash\Sigma_M$. Changing, if necessary, the constant $C(\gamma)$, the required inequality is fulfilled in the whole space.
\end{proof}
Now we prove Proposition \ref{prop_rad_global}
\begin{proof}
By Proposition \ref{prop_bounded}, $z_1(\delta)<u_\delta<z_3(\delta)$ and, by Remark \ref{rem_regularity}, $u_\delta$ is smooth. By Lemma \ref{lemma_dec_x'}, it converges to $z_3(\delta)$ as $|x|\to\infty$, therefore, by the famous symmetry result by \cite{GNN}, or by Theorem $2$ of \cite{FMR}, we conclude that $u_\delta$ is radially symmetric and radially decreasing. 
\end{proof}
Now we prove Proposition \ref{prop_uniqueness}.
\begin{proof}
Since, by Proposition \ref{prop_rad_global}, $u_\delta$ is radially symmetric and radially decreasing, then, up to translation, we have $u_\delta(0)=\min_{\R^N} u_\delta$. Since, by Lemma \ref{lemma_dec_x'}, $u_\delta(x)\to z_3(\delta)$ as $|x|\to\infty$, then it solves (\ref{prob_rad}), therefore, by uniqueness, $u_\delta=v_\delta$.
\end{proof}
In order to prove Proposition \ref{th_rad}, we need to apply Theorem $2$ of \cite{FMR}, which we recall, for the reader's convenience.
\begin{theorem}[\cite{FMR}]
Let $v>0$ be a bounded entire solution to 
$$-\Delta v=g(v)$$
in $\R^N$, with $g\in C^1(\R)$ such that $g'(s)\le 0$ in $(0,\eta)$, for some $\eta>0$. Writing $x=(y,z)\in\R^M\times\R^{N-M}$, we assume that
\begin{itemize}
\item $v(y,z)\to 0$ as $|y|\to\infty$, uniformly in $z$.
\item $v$ is periodic in $z$.
\end{itemize}
Then $v$ is radially symmetric in $y$, that is, up to a translation, $v(y,z)=w(|y|,z)$, and radially decreasing in $y$, that is $\partial_{y_j} v(y,z)<0$ for any $x=(y,z)\in\R^M\times\R^{N-M}$ with $y\ne 0$.\label{th_FMR}
\end{theorem}
\begin{proof}
By Proposition \ref{prop_bounded}, $z_1(\delta)<u_\delta<z_3(\delta)$ and, by Remark \ref{rem_regularity}, $u_\delta$ is smooth. By Lemma \ref{lemma_dec_x'}, it converges to $z_3(\delta)$ as $|x'|\to\infty$, uniformly in $x_N$. 
Since $u_\delta$ is periodic, in order to conclude that it is radially symmetric in $x'$ and radially decreasing, it is enough to apply Theorem \ref{th_FMR} to $v:=z_3(\delta)-u_\delta$. 
\end{proof}

\subsection{The asymptotic behaviour for $\delta$ small}
First we show that if a solution lies between $1/\sqrt{3}$ and $z_3(\delta)$, then it is constant. This is proved by the moving planes method.
\begin{lemma}
Let $\delta\in[0,2/3\sqrt{3})$ and let $u_\delta$ be a solution to (\ref{cahn-hilliard}) in $\R^N$ such that $u_\delta(x)\ge 1/\sqrt{3}$, for any $x\in\R^N$. Then $u_\delta\equiv z_3(\delta)$.
\label{lemma_const}
\end{lemma}
\begin{proof}
We set $v:=z_3(\delta)-u_\delta$. Setting, for any $\lambda\in\R$, $v_\lambda(x):=v(2\lambda-x_1,x'')$, we have 
\begin{eqnarray}
\label{start_moving_plane}
\text{$v-v_\lambda\ge 0$ in $\Sigma_\lambda$, for any $\lambda\in\R$.}
\end{eqnarray}
In order to prove this fact, we assume by contradiction that there exists $\lambda\in\R$ such that the open set $\Omega_\lambda:=\{x\in\Sigma_\lambda:v-v_\lambda<0\}$ is nonempty, and we observe that, in any connected component $\omega$ of $\Omega_\lambda$ we have
$$
\begin{cases}
-\Delta(v-v_\lambda)=h_\delta(v)-h_\delta(v_\lambda)<0 \qquad\text{in $\omega$,}\\
v-v_\lambda=0 \qquad\text{on $\partial\omega$,}
\end{cases}
$$
due to the strict monotonicity of $f_\delta$ in $[1/\sqrt{3},1)$ (for the definition of $h_\delta$, see the proof of Lemma \ref{lemma_exp_dec}). As a consequence, by the maximum principle for possibly unbounded domains, we have $v-v_\lambda\le 0$ in $\omega$, a contradiction.\\

By (\ref{start_moving_plane}), we have $\partial_{x_1}v\le 0$ in $\R^N$. The same argument applied to $\tilde{v}(x):=v(-x_1,x')$ implies that also $\tilde{v}$ satisfies (\ref{start_moving_plane}), hence $\partial_{x_1}v\ge 0$ in $\R^N$, thus $\partial_{x_1}v\equiv 0$. Composing $v$ with any rotation of $\R^N$, we conclude that $v$ is a constant solution to (\ref{cahn-hilliard}), thus $v\equiv 0$.
\end{proof}
Given the double well potential $W(t)=\frac{(1-t^2)^2}{4}$, $0<\alpha<W(1/\sqrt{3})=\frac{1}{9}$ and $\delta\in(0,2/3\sqrt{3})$, we set
$$\mu(\delta):=\max\{\mu<0:\, W_\delta(\mu)=\alpha\}.$$
Moreover, we take a smooth cutoff function $\chi:\R\to[0,1]$ such that $\chi=1$ in $(-\infty,-1)$ and $\chi=0$ in $(0,\infty)$ and we set 
\begin{equation}
\tilde{W}_\delta:=\chi_\delta \alpha+(1-\chi_\delta)W_\delta, \qquad \chi_\delta(t):=\chi(-t/\mu(\delta)).\label{def_tilde_W}
\end{equation}
We will denote $\tilde{W}:=\tilde{W}_0$. It is possible to see that $\tilde{W}_\delta$ enjoys the following properties:
\begin{equation}
\tilde{W}_\delta\to \tilde{W}, \qquad\text{as $\delta\to 0$, uniformly on compact subsets of $\R$,}
\end{equation}
\begin{equation}
\tilde{W}_\delta(r)=W_\delta(r), \qquad\text{for any $r\ge 0$ and $\delta\in[0,2/3\sqrt{3})$,}
\label{tildeW_right}
\end{equation}
and
\begin{equation}
\label{tildeW_unif_pos}
\inf_{(-\infty,0]}\tilde{W}= \alpha.
\end{equation}
In the sequel, we will be interested in a solution to 
\begin{equation}
\label{eq_beta}
\begin{cases}
-\Delta\beta_{R,\delta}+\tilde{W}_\delta(\beta_{R,\delta})=0 &\text{in $B_R$,}\\
\beta_{R,\delta}=z_1(\delta) &\text{on $\partial B_R$,}
\end{cases}
\end{equation}
for $\delta\ge 0$ small enough and $R$ large. This will be used as a barrier in the proof of Proposition \ref{prop_inf-sup}, which relies on a sliding method. This can be obtained in a variational technique, by minimising the functional
\begin{equation}
\label{def_functional}
J_{R,\delta}(v):=\int_{B_R}\bigg(\frac{1}{2}|\nabla v|^2+\tilde{W}_\delta(v)\bigg)dx.
\end{equation}
among all $H^1(B_R)$ functions with trace $z_1(\delta)$ on $\partial B_R$. The case $\delta=0$ is treated in Lemma $2.4$ of \cite{FV}. 
\begin{lemma}
Let $\delta_0>0$ be so small that $W_\delta(z_3(\delta))<\alpha/2$, for any $\delta\in[0,\delta_0)$. Then, For any $R>0$ and $\delta\in[0,\delta_0)$, there exists a minimiser $\beta_{R,\delta}\in C^2(B_R)$ of (\ref{def_functional}) among all functions with trace $z_1(\delta)$ on $\partial B_R$. Moreover, there exists $R_0>0$ such that, for any $R\ge R_0$ and for any $\delta\in[0,\delta_0)$, 
\begin{itemize}
\item 
\begin{equation}
\label{beta_bound}
z_1(\delta)<\beta_{R,\delta}(x)<z_3(\delta), \qquad\forall\, x\in B_R,
\end{equation}
\item 
\begin{equation}
\label{sup_beta}
\sup_{B_R}\beta_{R,\delta}>\frac{1}{\sqrt{3}},
\end{equation}
\item there exists a solution $\beta_R$ of (\ref{eq_beta}) with $\delta=0$ such that 
\begin{equation}
\label{sup_beta_delta_0}
\sup_{B_R}\beta_{R,\delta}\to\sup_{B_R}\beta_R\in[\frac{1}{\sqrt{3}},1) \qquad\text{as $\delta\to 0$}.
\end{equation}
\end{itemize}
\label{lemma_beta}
\end{lemma}
\begin{proof}
Existence follows from coercivity and weak lower semi continuity. By the fact that $\tilde{W}_\delta\equiv\alpha$ in $(-\infty,\mu(\delta))$ and (\ref{tildeW_right}), we can see the minimiser actually has to satisfy $z_1(\delta)\le\beta_{R,\delta}\le z_3(\delta)$, thus, due to the strong maximum principle, either (\ref{beta_bound}) holds or $\beta_{R,\delta}\equiv z_1(\delta)$.\\ 

Now we prove (\ref{sup_beta}), which, in particular, shows that $\beta_{R,\delta}> z_1(\delta)$ in $B_R$, at least for $R\ge R_0$. In order to do so, we assume that there exists a sequence $R_k\to\infty$ and a sequence $\delta_k\in[0,\delta_0)$ such that
$$\sup_{x\in\R^N}\beta_{R_k,\delta_k}\le \frac{1}{\sqrt{3}}.$$
It follows that, on the one hand
\begin{equation}
\label{energy_beta}
J_{R_k,\delta_k}(\beta_{R_k,\delta_k})\ge \alpha\omega_N R_k^N,
\end{equation}
where $\omega_N$ denotes the surface of $S^{N-1}$. On the other hand, if, for $R>1$ and $\delta\in[0,\delta_0)$, we take $w_{R,\delta}$ to be equal to $z_1(\delta)$ on $\partial B_R$ and to $z_3(\delta)$ in $B_{R-1}$ with $|\nabla w_{R,\delta}|$ bounded uniformly in $\delta$, we can see that there exists a constant $C>0$ such that, for $k$ large enough,
\begin{equation}
\label{energy_competitor}
J_{R_k,\delta_k}(w_{R_k,\delta_k})\le C R_k^{N-1}+W_{\delta_k}(z_3(\delta_k))\omega_N R_k^N<\alpha\omega_N R_k^N,
\end{equation}
since $\delta_k\in[0,\delta_0)$, hence $W_{\delta_k}(z_3(\delta_k))<\alpha/2$. This contradicts the minimality of $\beta_{R_k,\delta_k}$.\\

Finally we prove (\ref{sup_beta_delta_0}). In the forthcoming argument, $R>0$ will always be arbitrary but fixed. We observe that, since $\beta_{R,\delta}$ is bounded uniformly in $R>0$ and $\delta>0$, then any sequence $\delta_k\to 0$ admits a subsequence, that we still denote by $\delta_k$, such that $\beta_{R,\delta_k}$ converges in $C^2(B_R)$ to a solution $\beta_R$ to 
$$-\Delta \beta_R+\tilde{W}(\beta_R)=0\qquad\text{in $B_R$}$$
satisfying $\beta_R=-1$ on $\partial B_R$. Since the convergence is uniform and (\ref{beta_bound}) holds, then 
$$\sup_{B_R}\beta_{R,\delta}\to \sup_{B_R}\beta_R\in [-1,1]$$ 
as $\delta\to 0$. Moreover, by (\ref{sup_beta}) and the strong maximum principle, $\sup_{B_R}\beta_R\in [\frac{1}{\sqrt{3}},1)$.
\end{proof}
Now we can prove Proposition \ref{prop_inf-sup}.
\begin{proof}
It is enough to prove that, if there exists a sequence $\delta_k\to 0$, a sequence $u_{\delta_k}$ of solutions to (\ref{cahn-hilliard}) and $\nu>-1$ such that 
\begin{equation}
\label{deltak_inf}
\inf_{\R^N} u_{\delta_k}\ge \nu,
\end{equation} 
then there exists a subsequence $\delta_{k'}$ such that $u_{\delta_{k'}}\equiv z_3(\delta_{k'})$.\\

\textit{Claim: for any $\eps>0$ and $\rho>0$, there exists a subsequence, which we still denote by $u_{\delta_k}$, and a sequence $x^k\in\R^N$ such that} \\
\begin{equation}
u_{\delta_k}(x)>1-\eps, \qquad\forall\,x\in B_\rho(x^k).
\end{equation}

Since $\sup_{\R^N}u_{\delta_k}=z_3(\delta_k)$, there exists $x^k\in\R^N$ such that 
\begin{equation}
\label{almost_sup_points}
z_3(\delta_k)-u_{\delta_k}(x^k)<1/k.
\end{equation}
Therefore the sequence $u^k(x):=u_{\delta_k}(x+x^k)$ admits a subsequence converging, in $C^2_{loc}(\R^N)$, to a solution $u^\infty$ to the Allen-Cahn equation 
\begin{equation}
\label{allen-cahn}
-\Delta u^\infty=f(u^\infty), \qquad\text{in $\R^N$}.
\end{equation}
By (\ref{almost_sup_points}), we can see that $u^\infty(0)=1$, thus $u^\infty\equiv 1$. As a consequence, for any $\eps>0$ (small) and $\rho>0$, there exists a subsequence (still denoted by $u^k$) such that
$$\|u^k-1\|_{L^\infty(B_\rho)}<\eps, \qquad\forall k$$
hence the claim is true.\\

In order to prove our result, we first observe that, by (\ref{sup_beta}), for $\delta_0$ small as in Lemma \ref{lemma_beta} and $\delta\in(0,\delta_0)$, there exists $R>0$ and a solution $\beta_{R,\delta}$ to (\ref{eq_beta}) such that
\begin{equation}
\label{upper_sup_beta}
\sup_{B_R}\beta_{R,\delta}>\frac{1}{\sqrt{3}}, \qquad\forall\,\delta\in(0,\delta_0).
\end{equation}
Moreover, by (\ref{sup_beta_delta_0}), there exists a solution $\beta_R$ to 
$$-\Delta \beta_R+\tilde{W}(\beta_R)=0\text{ in $B_R$}, \qquad \beta_R=-1 \text{ on $\partial B_R$}$$
and $\delta_1=\delta_1(R)>0$ such that, for any $\delta\in(0,\delta_1)$, we have
\begin{equation}
\label{lower_sup_beta}
\sup_{B_R}\beta_{R,\delta}<\frac{\sup_{B_R}\beta_R +1}{2}<1, \qquad\forall\,\delta\in(0,\delta_1).
\end{equation}
As a consequence, for any $\delta\in(0,\bar{\delta})$, where $\bar{\delta}=\bar{\delta}(R):=\min\{\delta_0,\delta_1(R)\}$, we get
\begin{equation}
\frac{1}{\sqrt{3}}<\sup_{B_R}\beta_{R,\delta}<\frac{\sup_{B_R}\beta_R +1}{2}<1.
\end{equation}
Now, applying the claim with $\rho=R$ and
$$\eps:=1-\frac{\sup_{B_R}\beta_R +1}{2},$$
we can prove the existence of a subsequence, still denoted by $u_{\delta_k}$, and a sequence $x^k$ in $\R^N$ such that
$$u_{\delta_k}(x)>1-\eps>\sup_{B_R}\beta_{R,\delta_k}\ge \beta_{R,\delta_k}(x-x^k), \qquad\forall \, x\in B_R(x^k),\, \forall \, k.$$
Sliding $\beta_{R,\delta_k}$, with $k\ge k_0$ fixed, we get the lower bound 
$$u_{\delta_k}(x)> 1-\eps> \frac{1}{\sqrt{3}}, \qquad\forall\, x\in\R^N,\, \forall\, k\ge k_0.$$
In conclusion, by Lemma \ref{lemma_const}, $u_{\delta_k}\equiv z_3(\delta_k)$. 
\end{proof}

\begin{proposition}
Let $\delta\in(0,2/3\sqrt{3})$ and let $\{u_\delta\}_{\delta\in(0,\frac{2}{3\sqrt{3}})}$ be a family of non constant solutions to (\ref{cahn-hilliard}) in $\R^N$ such that
\begin{itemize}
\item for any $\delta\in(0,2/3\sqrt{3})$ there exists $R(\delta)>0$ such that $u_\delta>z_2(\delta)$ outside the cylinder $C_{R(\delta)}$.
\item $u_\delta$ is periodic in $x_N$.
\end{itemize}
Then 
\begin{equation}
\label{uto-1_compacts}
u_\delta\to -1 \qquad\text{as $\delta\to 0$, uniformly on compact subsets of $\R^N$.}
\end{equation}
and
\begin{equation}
\label{Rdelta_infty}
R(\delta)\to \infty \qquad\text{as $\delta\to 0.$}
\end{equation}
\label{prop_u_to-1}
\end{proposition}
\begin{remark}
We note that point (\ref{u_to-1}) of Theorem \ref{main_th} is a consequence of Proposition \ref{prop_u_to-1}.
\end{remark}
\begin{proof}
By Lemma \ref{lemma_BK}, the family $u_\delta$ is uniformly bounded, hence any sequence $\delta_k\to 0$ admits a subsequence, that we still denote by $\delta_k$, such that $u_{\delta_k}$ converges in $C^2_{loc}(\R^N)$ to a solution $u^\infty$ to the Allen-Cahn equation (\ref{allen-cahn}). Since $u_\delta$ are all non constant solutions, then, by Proposition \ref{prop_inf-sup}, we have
\begin{equation}
\inf_{\R^N}u_\delta\to -1, \qquad\text{as $\delta\to 0$.}
\end{equation}
By periodicity and Theorem \ref{th_rad}, we know that, for $\delta$ small, $u_\delta$ is radially symmetric in $x'$ and, up to a translation,
$$\inf_{\R^N}u_\delta=u_\delta(0),$$
hence, passing to the limit, we get $$u^\infty(0)=\lim_{k\to\infty}u_{\delta_k}(0)=\lim_{k\to\infty}\inf_{\R^N}u_{\delta_k}=-1,$$
which yields that $u^\infty\equiv -1$, thus (\ref{uto-1_compacts}) holds.\\

In order to prove (\ref{Rdelta_infty}), we assume by contradiction that there exists $\bar{R}>0$ and a sequence $\delta_k\to 0$ such that $R(\delta_k)\le \bar{R}$. By (\ref{uto-1_compacts}), $u_{\delta_k}\to -1$ uniformly in $B^{N-1}_{2\bar{R}}\times[-1,1]$, thus, for $k$ large enough, 
$$u_{\delta_k}(x'_k,0)<-\frac{1}{2}<z_2(\delta_k)$$
if, for instance, $x'_k=(2R(\delta_k),0)\in\R\times\R^{N-2}$, which contradicts the fact that $u_{\delta_k}$ is radially increasing.
\end{proof}

\end{document}